\documentclass[letterpaper, 10 pt, conference]{ieeeconf}

\IEEEoverridecommandlockouts
\usepackage{amsmath}
\usepackage{amssymb}
\usepackage{amsfonts}
\usepackage{graphicx}
\usepackage{nicefrac}

\newtheorem{theorem}{Theorem}

\newtheorem{corollary}{Corollary}

\newtheorem{assumption}{Assumption}

\usepackage{mathrsfs}

\usepackage{algorithm2e}

\usepackage{tikz}
\usetikzlibrary{shapes,arrows}
\usetikzlibrary{positioning}
\usetikzlibrary{plotmarks}

 \usepackage{pgfplots}
 \pgfplotsset{compat=newest}
 \pgfplotsset{plot coordinates/math parser=false}
 \pgfplotsset{label style={font=\small}}
 \pgfplotsset{legend style={font=\small}}
 \pgfplotsset{x tick label style={font=\small}}
 \pgfplotsset{y tick label style={font=\small}}

\newlength\figureheight
    \newlength\figurewidth
\newcommand{\R}{\mathbb{R}}
\newcommand{\N}{\mathbb{N}}

\newcommand{\PositiveMeasures}{\mathcal{M}^+}
\newcommand{\ubar}{\underline{u}}
\newcommand{\xbar}{\underline{x}}
\newcommand{\diff}{\mathrm{d}}

\author{Mathieu Claeys 
\thanks{M. Claeys is with Department of Engineering, University of Cambridge, Trumpington Street, Cambridge CB2 1PX, United Kingdom,
        \texttt{mathieu.claeys@eng.cam.ac.uk}}%
}

\title{Reconstructing trajectories from the moments of occupation measures}
\begin{document}
\maketitle

\begin{abstract}
Moment optimization techniques have been recently proposed to solve globally various classes of optimal control problems. As those methods return truncated moment sequences of occupation measures, this paper explores a numeric method for reconstructing optimal trajectories and controls from this data. In fact, by approximating occupation measures by atomic measures on a given grid, the problem reduces to a finite-dimensional linear program. In contrast with earlier numerical methods, this linear program is guaranteed to be feasible, no tolerance needs to be specified, and its size can be properly controlled. When combined with local optimal control solvers, this yields a powerful and flexible numerical approach for tackling difficult control problems, as demonstrated by examples.\\
\end{abstract}

Moments optimization techniques have emerged recently as a versatile tool for the global resolution of many non linear optimization problems, see for instance \cite{Lasserre2009Moments} and references therein for applications to polynomial optimization, optimal control, stochastic processes, and more. The general procedure is illustrated in Fig. \ref{fig:momentApproach}. First of all, the problem of interest is lifted or relaxed as a Linear Program (LP) on measures. As such, the problem becomes convex, albeit on a vector space which is hardly tractable in the general case, besides brute-force discretization. However, when problem data is polynomial, measures can be manipulated by their moments, which leads to a well studied hierarchy of moment relaxations, whose cost converges asymptotically to that of the measure LP.

\begin{figure}[b]
\centering

\begin{tikzpicture} [->,>=stealth', shorten >=1pt, node distance = 1.7cm, auto,
                     block/.style={rectangle, draw, thick, text width=7em, text centered, rounded corners, minimum height=3em}]
    \node [block] (opti) {Optimization \\ problem };
    \node [block, below of=opti] (GMP) {LP on \\ measures };
    \node [block, below of=GMP] (GMPmoms) {LP on \\ moments };
    \node [block, below of=GMPmoms] (relax) {Semi-definite \\  relaxations};
    \node [left of=opti]   (ghostL1) {};
    \node [left of=relax](ghostL2) {};
    \node [right of=opti]   (ghostR1) {};
    \node [right of=relax](ghostR2) {};
    \node [below left of=relax](ghostB1) {};
    \node [below right of=relax](ghostB2) {};

    \path (opti) edge (GMP);
    \path (GMP) edge  (GMPmoms);
    \path (GMPmoms) edge (relax);
    \path (ghostL1)  edge node[left, rotate=90, anchor=south] {Convex lift} (ghostL2);
    \path (ghostR2)  edge node[right, rotate=90, anchor=north] {Inverse problem} (ghostR1);
    \path (ghostB1)  edge node[above] {Numeric optimization} (ghostB2);
    
\end{tikzpicture}
\caption{The moment approach. An optimization problem is relaxed as semi-definite relaxations, solved numerically. The solution of the original problem is found by solving an inverse problem from the optimized data.}
\label{fig:momentApproach}
\end{figure}
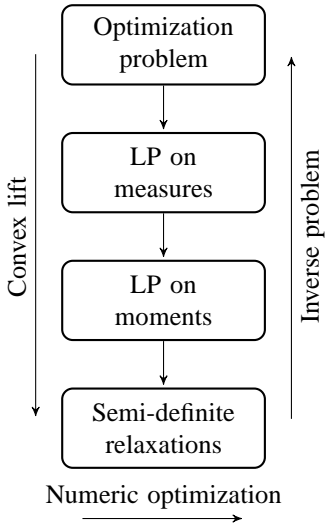

After numeric resolution of one of those relaxations, it is of obvious interest to assert termination of the hierarchy, as well as extract from moment data the solution of the measure LP, which in turn must be interpreted in terms of solutions of the original problem. This last step of the method is a typical example of an inverse problem.

For polynomial, finite-dimensional optimization as proposed in \cite{Lasserre2000optiGlob}, the full approach is now well mastered. From the optimized moment data, one can expect generically at a finite relaxation order \cite{Nie2012hierarchy} to recover the support of the measure, which is located at the (generically unique) optimal point. Numerical tools allow for the resolution of such polynomial problems by the mere definition of the problem (see for instance \cite{GloptiPoly}), for small size problems. 

For Optimal Control Problems (OCP), on the other hand, most recent work has been focused on the convex lift and optimization step of the approach, for various classes of systems: from the case of bounded controls as presented in Lasserre et al. \cite{Lasserre2008NLOCP}, the approach was subsequently extended to impulsive controls \cite{impulseTAC} or switched systems \cite{Henrion2013Switched}. However, very little work has been done on the inverse problem, which precludes the same ``black box'' operations as for polynomial optimization. To the author's knowledge, only exploitation of the dual problem as hinted in \cite{Lasserre2008NLOCP} or the resolution of an inverse problem with very limited moment data \cite{mevissen} have been investigated. We present briefly those methods and compare them to the one presented on this paper in \S\ref{sec:example}.

Several obstacles related to the nature of solutions of OCPs explain such a gap.
First of all, trajectories are inherently infinite-dimensional objects,  such that one can only expect to have approximate moments to the measure LP as a result of any given \emph{finite order} moment relaxation. In addition, even if those moments were exact, their finite number prevents the exact reconstruction of an infinite dimensional object. Therefore, only approximate optimal trajectory and control law can be recovered. This is especially true for the control, which is expected to be highly discontinuous, switching in between bang, path-constrained and singular arcs.

Due to the restrictions mentioned above, we propose a numeric approach to reconstruct trajectories. In much of the spirit of the moment approach itself, we propose a convex program for extracting trajectory/control way points. As such, the scheme is guaranteed to converge if an admissible solution exists, and users should not specify any starting point for the method. The method is based on the approach developed by Rubio (see \cite{Rubio1986Book} and the references therein), which uses a linear program to solve the measure LP directly. However, the key difference here is that in our approach, only an inverse problem must be solved, whereas the actual trajectory optimization is performed through semi-definite relaxations. The segregation between optimization and inverse problem allow the overcome several difficulties arising in those earlier works.

As a last step, the obtained way points are used as starting data of a direct optimal control routine. This allows for the reconstructing of a trajectory to any given precision. As a result, this also equips to approach with a termination criterion; At any given relaxation order, if the cost of the solution of the direct method agrees with some prescribed tolerance to the cost of the moment relaxation, the solution is certified as globally optimal.

The idea of using a global method for hot-starting local optimal control methods is surprisingly not well studied in the literature. We explain this with the fact that the numeric resolution of the Hamilton-Jacobi-Bellman (HJB) equation is an endeavor that can be hardly undertaken gradually in function of computational power available, as bounds on the discretization grid must be enforced to avoid numerical instability. In addition, once the solution is available, only those applications requiring high-precision solutions truly need a second  round of computations from a local method. This is the motivation behind Cristiani and Martinon \cite{Cristiani10initialization}, who propose to solve the HJB equation on a coarse grid to initialize an indirect method exploiting Pontryagin's maximum principle \cite{PMP}, via their well-known relationship (see e.g. \cite{Cernea} in the constrained case).

In contrast with these limitations, moment relaxations of low order are very cheap to compute, and as shown in the examples, are usually tight enough to hot start a local method in a sufficiently close neighborhood of the optimal solution. Up to the author's knowledge, this is the first time that moment relaxations are used for such a purpose.

\section*{Notations}

Let $\mathbf{Z} \in \R^n$ be a compact set of an Euclidean space. We note by $\PositiveMeasures(\mathbf{Z})$ the space of finite, positive measures supported on $\mathbf{Z}$, equipped with the weak-$*$ topology.
For a continuous function $f(z) \in C(\mathbf{Z})$, denote by $\int_{\mathbf{Z}} \! f(z) \, \mu(dz)$ the integral of $f(z)$ by the measure $\mu \in \mathcal{M}^+(\mathbf{Z})$.
When no confusion may arise, we note $\langle f, \mu \rangle$ for the integral to simplify exposition and to insist on the duality relationship between continuous functions and measures. The Dirac measure supported at $z^*$ is denoted by $\delta_{z^*}$.

For multi-index $\alpha \in \N^n$ and vector $z\in \R^n$, we use the notation $z^\alpha := \prod_{i=1}^n z_i^{\alpha_i}$. Denote by $\N_m^n$ the set $\lbrace \alpha \in \N^n : \; \sum_{i=1}^n \alpha_i \leq m \rbrace$.
The moment of multi-index $\alpha \in \N^n$ of measure $\mu \in \PositiveMeasures(\mathbf{Z} \subset \R^n)$ is then defined as the real $y_\alpha = \langle z^\alpha, \mu \rangle$. A multi-indexed sequence of reals $\lbrace y_\alpha\rbrace_{\alpha \in \N^n}$ is said to have a \emph{representing measure} on $\mathbf{Z}$ if there exists $\mu \in \PositiveMeasures(\mathbf{Z})$ such that $y_\alpha = \langle z^\alpha , \mu \rangle$ for all $\alpha \in \N^n$.

Denote by $\R[z]$ the ring of polynomials in the variables $z$. A set $\mathbf{Z} \in \R^n$ is basic semi-algebraic if it is defined as the intersection of finitely many polynomial inequalities: $\mathbf{Z} := \lbrace z \in \R^n : \; g_i(z) \geq 0, \, g_i(z) \in \R[z], \, i = 1 \ldots n_{\mathbf{Z}}\rbrace$.

Finally, we use the notation $\xbar$ to denote parameters of the state space where trajectories $x(t)$ live. We use the same convention $\ubar$ for the controls $u(t)$. This notation makes the passage from temporal integration to integration with respect to a measure transparent.

\section{Moment optimization for optimal control} \label{sec:OCP}
This section outlines the main steps of the convex lift step of the moment approach for bounded OCPs as presented in \cite{Lasserre2008NLOCP}. We only highlight the important features necessary for the following sections.

Consider the following end-constrained\footnote{Note that the end-point constraints considered here are for ease of exposition, see \cite{Lasserre2008NLOCP} for the general case.} problem:
\begin{equation}
\label{eq:OCP}
\begin{aligned}
J = \inf_{u(t)} \; & \int_{t_i}^{t_f} \!\!\! h(t,x,u) \, \diff t \\
\text{s.t.} \; & \dot{x} = f(t,x,u), \\
& x(t_i), \; x(t_f) \; \text{given}, \\
& x(t) \in \mathbf{X}, \quad u(t) \in \mathbf{U}, \\
& t \in \mathbf{T} := [t_i,t_f],
\end{aligned}
\end{equation}
where the state $x(t) \in \R^n$, controls $u(t) \in \R^m$ and functions $f,h \in \R[t,\xbar,\ubar]$. Sets $\mathbf{X}$ and $\mathbf{U}$ are compact, basic semialgebraic sets, and are assumed w.l.g. to be contained in a ball included in the algebraic definition of those sets.

For a process $(u(t),x(t))$ admissible for \eqref{eq:OCP}, the \emph{occupation measure} $\mu \in \mathcal{M}^+([\mathbf{T} \times \mathbf{U} \times \mathbf{X})$ is defined by:
\begin{equation}
\label{eq:occMeas}
    \mu(\mathbf{A} \times \mathbf{B} \times \mathbf{C}) := \int\limits_{ \mathbf{A}} \!\! \delta_{u(t)}(\mathbf{B}) \, \delta_{x(t)}(\mathbf{C}) \; \mathrm{d}t,
\end{equation}
where $\mathbf{A}$, $\mathbf{B}$ and $\mathbf{C}$ are Borel subsets of resp. $\mathbf{T}$, $\mathbf{U}$ and $\mathbf{X}$. 

When evaluating a continuously differentiable test function $v(t,\xbar)$ along such an admissible trajectory, straightforward computations reveal that the occupation measure is admissible for the following measure LP, while achieving the same cost as \eqref{eq:OCP}:
\begin{equation}
\label{eq:measureLP}
\begin{aligned}
J_{LP} = \inf_{\mu} \; & \langle h, \mu \rangle \\
\text{s.t.} \; & \forall v \in \R[t,\xbar]: \; \left[ v( \cdot, x(\cdot))\right]_{t_i}^{t_f} = \langle \frac{\partial v}{\partial t} + \frac{\partial v}{\partial \xbar}  f , \mu \rangle, \\
& \mu \in \PositiveMeasures(\mathbf{T} \times \mathbf{U} \times \mathbf{X}).
\end{aligned}
\end{equation}
Obviously, since the admissible elements have been enlarged, $J_{LP} \leq J$. It is however expected generically that $J_{LP} = J$ hold, see e.g. the discussion in \cite{roa}.

As problem data was assumed polynomial, viz. $f,h \in \R[t,\xbar,\ubar]$ and $\mathbf{X}$ as well as $\mathbf{U}$ are basic, semi-algebraic set, a dual to Putinar's theorem (see \cite[Th. 3.8b]{Lasserre2009Moments}) allows to handle measures by their moment sequences. One obtains a finite-dimensional convex relaxation of \eqref{eq:measureLP} by truncating the problem with only the first few moments as decision variables, as well as only a few of the linear constraints. Denoting by $J_{SDP}^r$ the cost of the semi-definite relaxation considering moments of degree up to $2r$, \cite{Lasserre2008NLOCP} have proven that the cost of the relaxations converge from below to the cost of measure LP \eqref{eq:measureLP}:

\begin{theorem}[Lasserre et al.]
\begin{equation}
    \lim_{r \rightarrow \infty} J_{SDP}^{\, r} \uparrow J_{LP}.
\end{equation}
\end{theorem}

The following corollary is immediate, and provides the basis for the termination criterion of \S\ref{sec:inverse}:
\begin{corollary}
\label{th:termination}
Let an admissible pair $(u(t),x(t))$ have the same cost as a moment relaxation. Then the solution is globally optimal.
\end{corollary}

For the rest of the paper, we make the following standing assumptions, to ease exposition:
\begin{assumption}
\label{th:uniqueness}
Problem \eqref{eq:OCP} admits a unique optimal process.
\end{assumption}
Note that assumption \ref{th:uniqueness} is generic, and can be enforced almost surely by perturbing randomly the coefficients of the polynomials data, as done for polynomial optimization in the \texttt{SparsePOP} toolbox \cite{SparsePOP}. Note that as an immediate corollary, Assumption \ref{th:uniqueness} imposes the uniqueness of the optimal measure by \cite[Cor. 1.4]{Vinter1993convexDuality}, and since measures are moment determinate on compact sets, uniqueness of optimal moment sequences are guaranteed as well. The next sections explore how this unique optimal measure can be reconstructed from its moments data. That is, given definition \eqref{eq:occMeas}, we propose an algorithm for reconstructing the approximate support of $\mu$.

\section{Atomic approximation of occupation measures}
\label{sec:atomic}

This section outlines convergence results for approximating occupation measures by atomic measures, that is, measures supported on a finite number of points only. In the next sections, it is shown how trajectories can be recovered from these approximate measures.

At any given relaxation, only a finite subset of the moment constraints of \eqref{eq:measureLP} can be taken into account numerically. As a result, if the optimized moment sequence possesses a representing measure, there always exists an atomic measure with the exact same moments \cite[Th. B.12]{Lasserre2009Moments}:
\begin{theorem}[Tchakaloff]
\label{th:Tchaka}
Let $\mu$ be a finite, positive, Borel measure with compact support $\mathbf{Z} \subset \R^q$, and let $d\geq 1$ be a fixed positive integer. Then there exists $p \leq \binom{ q + d }{ q }$ points $z_k \subset \mathbf{Z}$ and positive weights $w_k$ such that
\begin{equation}
\langle f, \mu \rangle = \sum_{k=1}^p w_k \, f(z_k)
\end{equation}
for every polynomial $f \in \R[z]$ of degree at most $d$.
\end{theorem}

To obtain a finite dimensional LP, we therefore propose fixing a time-space grid of Dirac measures, as proposed in \cite{Rubio1986Book} for solving \eqref{eq:measureLP} directly. As such the only decision variables are then the mass of the Dirac measures, which enter linearly in the problem. The crucial difference here with \cite{Rubio1986Book} is that we do not attempt to find an optimal measure, since its (approximate) moments are already given by the semi-definite relaxations. Instead, by segregating optimization and approximation, we circumvent the problem encountered in \cite{Rubio1986Book}, where there might not be a given atomic measure on a given grid satisfying the truncated moment constraints within a given tolerance. We also prevent the distortion of solutions when prescribed moment tolerances are too loose.

Denote by $\mathbf{Z}_\varepsilon \subset \mathbf{Z} \subset \mathbf{T} \times \mathbf{U} \times \mathbf{X}$ a mesh of given resolution $\varepsilon$ supported on finitely many points, such that for all $z \in \mathbf{Z}$, there exists $z_i \in \mathbf{Z}_\varepsilon$ such that $\lvert z_i - z \rvert \leq \varepsilon$. By compactness of $\mathbf{Z}$, such a $\mathbf{Z}_\varepsilon$ always exists.

\begin{theorem}
Consider the following LP:
\begin{equation}
\label{eq:approxLP}
\begin{aligned}
\lambda_\varepsilon^* = \min_{\tilde{\mu}, \lambda}  \; & \lambda \\
\text{s.t.} \; &   \vert {y}_\alpha - \langle z^\alpha, \tilde{\mu} \rangle \rvert \leq \lambda, \quad \forall \alpha \in \N^{1+m+n}_{2r}\\
& \tilde{\mu} \in \PositiveMeasures ( \mathbf{Z}_\varepsilon ),
\end{aligned}
\end{equation}
where ${y}_\alpha$ are given moments of a representing measure $\mu \in \PositiveMeasures(\mathbf{Z})$.
Then, as the mesh is refined, moments of the approximate atomic measure $\tilde{\mu}$ converge to those of $\mu$:
\begin{equation}
\lim_{\varepsilon \rightarrow 0} \lambda_\varepsilon^* \rightarrow 0.
\end{equation}
\end{theorem}

\begin{proof}
Indeed, Th. \ref{th:Tchaka} asserts the existence  of an atomic measure supported on $\binom{1+n+m+2r}{2r}$ points $\lbrace z_k \rbrace$ such that $y_\alpha = \sum_k z_k^\alpha$, for all $\alpha \in \N_{2r}^{1+n+m}$. Therefore, for any $\lambda > 0$, by continuity of the monomials and the finite number of moment constraints, one can always find an $\epsilon > 0$ such that there exists a measure $\tilde{\mu} \in \PositiveMeasures(\mathbf{Z}_\epsilon)$ that is admissible for LP \eqref{eq:approxLP}. As $\lambda$ is arbitrary and 0 is a lower bound for the cost, this concludes the proof.
\end{proof}
See also the related developments in \cite[\S{}3]{Rubio1986Book}, which however try to approximate \eqref{eq:measureLP} directly by a LP. To optimize over the cost, moment tolerance $\lambda$ must then be fixed a priori, and carefully chosen to minimize distortions while guaranteeing the existence of an atomic approximation on the given grid -- a very difficult task.

Problem \eqref{eq:approxLP} can be interpreted as a truncation of a minimum norm problem, when the weak-$*$ topology is considered for measures. Indeed, on compact sets, it is sufficient to consider a dense basis of the continuous function equipped with the supremum norm. The monomials $z^\alpha$ are such a possible choice, such that the minimum norm problem would read:
\begin{equation}
\label{eq:weakMinNorm}
\inf_{\tilde{\mu} \in \PositiveMeasures ( \mathbf{Z}_\varepsilon )}
\; \sup_{\alpha \in \N^{1+n+m}} \lvert \langle z^\alpha ,  \tilde{\mu} - \mu \rangle \rvert.
\end{equation}
Clearly, \eqref{eq:approxLP} is just the truncation of \eqref{eq:weakMinNorm} down to a finite number of moments. This has an immediate consequence by the mere definition of weak-$*$ convergence, as relaxation order $r$ in \eqref{eq:approxLP} is increased:
\begin{corollary}
\label{th:assympConv}
$\tilde{\mu}$ converges weakly-$*$ to $\mu$ as $r \rightarrow \infty$ and $\varepsilon \rightarrow 0$.
\end{corollary}

\section{Practical inverse problem}
\label{sec:inverse}

Corollary \ref{th:assympConv} shows the good asymptotic properties of the algorithm as the grid is refined and the number of moments is increased. Obviously, only a finite number of moments are supplied by any given moment relaxation, and the practical resolution of the LP imposes a maximum size on the approximation grid. From now on, we take a practical view on the inverse problem, working with a finite data set, and with possibly approximate/relaxed moment data as problem input. In the next section, this standing assumption of ``good enough'' moments and grid resolution will be tested a posteriori by hot-starting a direct method.

Fix a grid $\mathbf{Z}_\varepsilon$ indexed by $q$ points. Then \eqref{eq:approxLP} is rewritten as the following minimum norm problem parametrized by the (positive) weights $w_\beta$ associated to each atom $z_\beta \in \mathbf{Z}_\varepsilon, \beta=1, \ldots, q$:
\begin{equation}
\label{eq:approxLPexplicit}
\lambda_\epsilon^* = \min_{ \displaystyle w \in \R_+^q}  \; \lVert b - A \, w \rVert_\infty,
\end{equation}
with vector $b$ the truncated moment vector $\lbrace y_\alpha \rbrace$, $\alpha \in \N_{2r}^{1+m+n}$ returned by the semi-definite relaxation of order $r$, and $A$ a $\N_{2r}^{1+m+n} \times q$ matrix, whose elements are given by $A_{\alpha \beta} = z_\beta^\alpha$. This problem can be solved by standard linear programming routines. The support of the atomic measure -- which is expected to approximate closely the support of occupation measure \eqref{eq:occMeas} -- is recovered by the elements of $\mathbf{Z}_\varepsilon$ whose optimal weights are non-zero. By Ass. \ref{th:uniqueness}, those way-points closely approximate the unique optimal trajectory and controls. Note that for the numeric experiments below, it was observed to be more efficient to solve \eqref{eq:approxLPexplicit} by an interior point method and apply a threshold than to use a simplex-based algorithm.

There is nonetheless a severe restriction to this approach: the dimension of $\mathbf{T} \times \mathbf{X} \times \mathbf{U} \subset \R^{1+n+m}$ leads quickly to LPs of extremely large sizes, a direct consequence of the ``curse of dimensionality'' of dynamic programming, the conic dual of \eqref{eq:measureLP}. Indeed, in the opinion of the author, this is the second main restriction for using the numeric approach of \cite{Rubio1986Book} for solving \eqref{eq:measureLP} directly. Even with a coarse grid of $100$ points in each direction, a basic implementation of the algorithm would lead to a LP of a minimum of $10$ billion variables for a problem with $1$ control and $3$ states. Therefore, we propose a middle ground between the rigor and the better scalabilty of the moment approach, and the practical information on the support of the occupation measure given by the LP approach. From the data given by a relaxation, we only consider moments of the form $y_{i 0 \ldots 0 j 0 \ldots 0}$. That is we, use only moments of time and one of the states or the controls, such as to perform a coordinate-by-coordinate identification of trajectory and control time series. This way of proceeding keeps the size of the LPs small enough, such that the inverse problem has a negligible resolution time in comparison to that of the moment relaxation, while maintaining a sufficient resolution $\varepsilon$ of the grid.

%

Obviously, although likely to give excellent results even in the face of approximated and truncated data, this procedure does not guarantee admissibility of the reconstructed states and controls for OCP \eqref{eq:OCP}. To refine the numeric solutions, we propose to use those approximate time-series to hot-start a direct optimal control method. For all the numeric tests of the next section, we used the freely available software \texttt{BOCOP} \cite{BOCOP}. Corollary \ref{th:termination} then provides a numeric termination criterion for the moment method: if the cost of the solution given by the direct method is within a prescribed tolerance of the cost of the moment relaxation, the local solution is validated as globally optimal and the hierarchy of moment relaxations can be terminated. However, as shown in the next section, even non-tight relaxations are already enough to recover the global optimal solution via this approach, although those solutions can obviously not be certified to be optimal.

\section{Illustrative examples}
\label{sec:example}

We present now three illustrative examples. Moment relaxations were solved via \texttt{GloptiPoly} \cite{GloptiPoly}, using \texttt{SeDuMi} \cite{sedumi} as the semi-definite solver.

\subsection{Double integrator}
\label{sec:doubleint}

Consider the problem of driving the double integrator to the origin in minimum time:
\begin{equation}
\label{eq:doubleint}
\begin{aligned}
J = \inf_{u(t)} \; & t_f \\
\text{s.t.} \; & \begin{bmatrix} \dot{x}_1 \\ \dot{x}_2 \end{bmatrix} = \begin{bmatrix} 0 \\ u \end{bmatrix}, \\
& x(0) = \begin{bmatrix} 1 \\ 1 \end{bmatrix}, \; x(t_f) = \begin{bmatrix} 0 \\ 0 \end{bmatrix}, \\
& u(t) \in [-1,1].
\end{aligned}
\end{equation}

Figures \ref{fig:doubleintx1} to \ref{fig:doubleintu} present the optimal solution, along with the reconstructed trajectory following a coordinate-by-coordinate identification as proposed in \S\ref{sec:inverse}. Moments of the fourth-order relaxation were used, since they closely approximate the true moments of the optimal occupation measure. As one can see, the reconstructed trajectories and controls closely match the true solution. 

\begin{figure}
\centering
\input{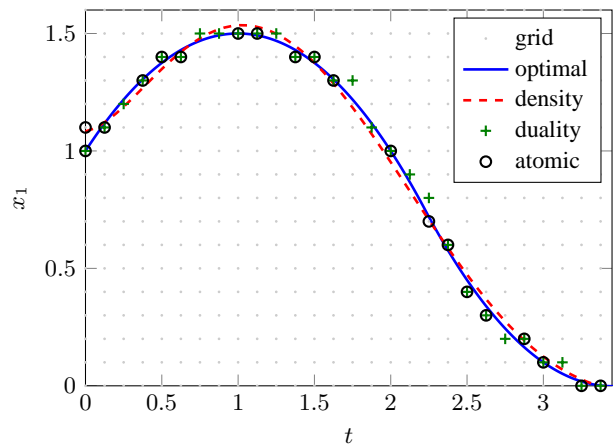}
\caption{Reconstruction of state $x_1$ of the example of \S\ref{sec:doubleint}. Optimal solution versus reconstructions via polynomial densities , dual HJB exploitation and atomic approximations.}
\label{fig:doubleintx1}
\end{figure}

\begin{figure}
\centering
\input{Figures/x2.tikz}
\caption{Reconstruction of state $x_2$ of the example of \S\ref{sec:doubleint}. Optimal solution versus reconstructions via polynomial densities , dual HJB exploitation and atomic approximations.}
\label{fig:doubleintx2}
\end{figure}

\begin{figure}
\centering
\input{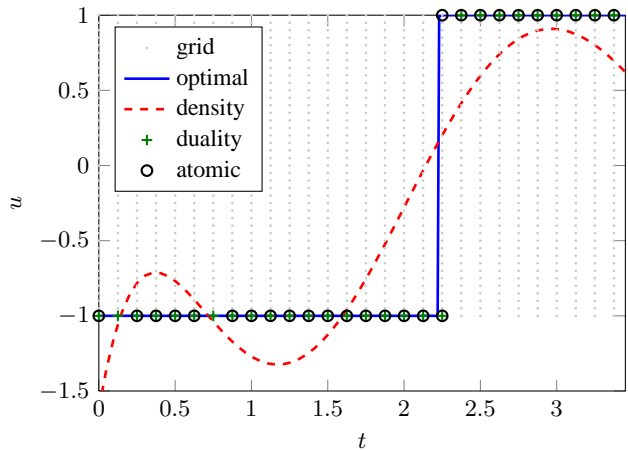}
\caption{Reconstruction of control $u$ of the example of \S\ref{sec:doubleint}. Optimal solution versus reconstructions via polynomial densities , dual HJB exploitation and atomic approximations.}
\label{fig:doubleintu}
\end{figure}

Figures \ref{fig:doubleintx1} to \ref{fig:doubleintu} also compare the proposed atomic reconstruction approach with methods based on the results of \cite{Lasserre2008NLOCP} and \cite{mevissen}. The approach hinted in \cite{Lasserre2008NLOCP} consists in extracting dual solutions to each moment relaxation, which give a polynomial (sub)approximation $V$ to the value function satisfying, on $\mathbf{T} \times \mathbf{X} \times \mathbf{U}$,
\begin{equation}
\label{eq:compl}
\frac{\partial V}{\partial t} + \frac{\partial V}{\partial x}  f + h \geq 0.
\end{equation}
Around the optimal trajectories, equality in \eqref{eq:compl} must hold. We therefore propose a very coarse exploitation of such results: for each time $t_i$ in a given discretization of the time interval, find the couple $(u(t_i), x(t_i))$ minimizing  $\frac{\partial V}{\partial t} + \frac{\partial V}{\partial x}  f + h$. We did so by exhaustive search on the same time-state-control grid used by the atomic approximation technique, for comparison purposes. Obviously, the main drawback of the approach is the curse of dimensionality, the example presented with $1$ control and $2$ states being already computationally intensive. The method gives excellent result for reconstructing the state, although slightly less precise than the faster atomic approximation.

Finally, the approach in \cite{mevissen} considers moments of the form $y_{i 0 \ldots 0 1 0 \ldots 0}$ to approximate occupation  measures by absolutely continuous (with respect to the Lebesgue measure on time) measures with polynomial densities. This results in a simple linear system to solve. As can be seen in Fig.~\ref{fig:doubleintx1}, the method allows for a correct approximation of the first state, while the second state in Fig.~\ref{fig:doubleintx2} cannot capture properly the switch of velocity. The picture is worse for the control, as seen in Fig.~\ref{fig:doubleintu}, as the polynomial density fails to identify any structure such as bang arcs, and gives controls below $-1$, hence non-admissible. Such behavior is of course expected from a polynomial approximation of a discontinuous function. In addition, this method simply fails if several solutions are encoded in an occupation measures, since the recovered density will be a weighted mean of those solutions. On the other hand, the atomic approximation measure will be supported on all these trajectories. As long as their number is finite, they could be extracted form the approximate measures, albeit with a more sophisticated technique to identify optimal arcs individually.

%
%

As a summary, the atomic approximation combines computational effectiveness with a good approximation quality, especially considering admissibility of reconstructed trajectories and detection of control structures.

\subsection{Non-convex integrator}
\label{sec:nonconvex}
Consider the problem:
\begin{equation}
\label{eq:nonconvexint}
\begin{aligned}
J = \inf_{u(t)} \; & \int_0^1 \!\! x^2 \, \diff t \\
\text{s.t.} \; & \dot{x} = u, \\
& x(0) = 0, \quad x(1) = \nicefrac{3}{4}, \\
& x(t) \in [-1,1], \\
& u(t) \in [1,1], \\
& \left( x(t) - \nicefrac{1}{5} \right)^2 + \left( t - \nicefrac{1}{2} \right)^2 \geq \left( \nicefrac{1}{5} \right)^2.
\end{aligned}
\end{equation}
The last constraint represents a time-dependent obstacle to be avoided, with the cost favoring solution passing below it. However, these solutions cannot be admissible, as the end-point constraints will not be satisfied, and the optimal solution necessarily  passes above the obstacle (see Fig.~\ref{fig:nonconvex}). This interplay between non-convex constraints and the cost make this problem very difficult to initialize for local methods, as a few tries with solver \texttt{BOCOP} \cite{BOCOP} may reveal.

We now show that the method proposed can be used efficiently to initialize the local method. Figure \ref{fig:nonconvex} presents the extracted trajectories from moments of the second relaxation (cost of $0.141$), the sixth relaxation (cost of $0.164$) and the optimal solution (cost of $0.176$). As expected, the sixth relaxation offers a finer reconstruction of the optimal trajectory. However, it should be noted that the very coarse trajectory reconstructed from the second relaxation is already enough to initialize the direct method implemented in \texttt{BOCOP} so that it converges to the global optimum. This suggests that the method proposed in this paper could be widely used for hot starting local methods, as lower-order relaxations are very cheap to compute. This is in contrast with using Hamilton-Jacobi-Bellman solvers as in \cite{Cristiani10initialization}, which offer little possibility of trading solution accuracy with computational load.

\begin{figure}
\centering
%
%
%
\definecolor{mycolor1}{rgb}{0.00000,0.49804,0.00000}%
\begin{tikzpicture}

\begin{axis}[%
width=\figurewidth,
height=\figureheight,
scale only axis,
xmin=0,
xmax=1,
xlabel={t},
ymin=-0.1,
ymax=0.8,
ylabel={x(t)},
legend style={at={(0.03,0.97)},anchor=north west,draw=black,fill=white,legend cell align=left}
]
\addplot [color=red,dashed,line width=1.0pt]
  table[row sep=crcr]{0.7	0.2\\
0.699597335294377	0.212684783931313\\
0.698390962566159	0.22531849071475\\
0.696385739452541	0.237850248872082\\
0.693589740279271	0.250229597436216\\
0.690014223548189	0.262406689139697\\
0.685673586603215	0.274332491132066\\
0.680585307657324	0.285958982417834\\
0.674769875413957	0.297239347220094\\
0.668250706566236	0.30812816349112\\
0.661054051506212	0.318581585810928\\
0.653208888623796	0.328557521937308\\
0.644746807621014	0.338015802296422\\
0.635701882311426	0.346918341731507\\
0.626110533416904	0.355229292858351\\
0.61601138191424	0.362915190410067\\
0.605445093522101	0.369945085989903\\
0.594454214954537	0.376290672689516\\
0.583083002600377	0.381926399070904\\
0.571377244318374	0.386829572053021\\
0.559384075065655	0.390980448288815\\
0.547151787101886	0.394362313664708\\
0.534729635533386	0.396961550602442\\
0.522167639980202	0.398767692892251\\
0.509516383164748	0.399773467836602\\
0.496826807233038	0.399974825534775\\
0.484150008628642	0.399370955190389\\
0.471537032345343	0.397964288376187\\
0.459038666386962	0.395760489242956\\
0.446705237261993	0.392768431711988\\
0.434586407336516	0.389000163742934\\
0.422730974861374	0.384470858820916\\
0.411186677478845	0.379198754858267\\
0.4	0.373205080756888\\
0.389215987226778	0.366513970926954\\
0.378878062572467	0.359152368106166\\
0.369027853210943	0.351149914870852\\
0.359705022458736	0.342538834275773\\
0.350947110064849	0.333353800103258\\
0.342789381051443	0.323631797244121\\
0.335264683714033	0.313411972772554\\
0.328403317353005	0.302735478314681\\
0.322232910269015	0.291645304345482\\
0.316778308513586	0.280186107081323\\
0.312061475842818	0.268404028665134\\
0.308101405277101	0.256346511368286\\
0.304914042622919	0.244062106557308\\
0.302512222264721	0.23160027919467\\
0.300905615485383	0.219011208660837\\
0.300100691523363	0.206345586699614\\
0.300100691523363	0.193654413300386\\
0.300905615485383	0.180988791339163\\
0.302512222264721	0.16839972080533\\
0.304914042622919	0.155937893442692\\
0.308101405277101	0.143653488631714\\
0.312061475842818	0.131595971334866\\
0.316778308513586	0.119813892918677\\
0.322232910269015	0.108354695654518\\
0.328403317353005	0.0972645216853187\\
0.335264683714033	0.0865880272274459\\
0.342789381051442	0.076368202755879\\
0.350947110064849	0.0666461998967417\\
0.359705022458736	0.0574611657242275\\
0.369027853210943	0.0488500851291483\\
0.378878062572467	0.0408476318938336\\
0.389215987226778	0.0334860290730457\\
0.4	0.0267949192431123\\
0.411186677478845	0.0208012451417328\\
0.422730974861374	0.0155291411790837\\
0.434586407336516	0.0109998362570663\\
0.446705237261993	0.00723156828801161\\
0.459038666386962	0.00423951075704426\\
0.471537032345343	0.00203571162381347\\
0.484150008628642	0.00062904480961154\\
0.496826807233038	2.51744652249863e-005\\
0.509516383164748	0.000226532163398407\\
0.522167639980202	0.00123230710774919\\
0.534729635533386	0.00303844939755837\\
0.547151787101885	0.00563768633529163\\
0.559384075065655	0.00901955171118524\\
0.571377244318374	0.0131704279469786\\
0.583083002600377	0.0180736009290963\\
0.594454214954536	0.0237093273104835\\
0.605445093522101	0.0300549140100971\\
0.61601138191424	0.0370848095899328\\
0.626110533416904	0.0447707071416486\\
0.635701882311426	0.0530816582684933\\
0.644746807621014	0.0619841977035775\\
0.653208888623795	0.071442478062692\\
0.661054051506212	0.081418414189072\\
0.668250706566236	0.0918718365088805\\
0.674769875413957	0.102760652779906\\
0.680585307657324	0.114041017582166\\
0.685673586603214	0.125667508867934\\
0.690014223548189	0.137593310860303\\
0.693589740279271	0.149770402563784\\
0.696385739452541	0.162149751127918\\
0.698390962566159	0.17468150928525\\
0.699597335294377	0.187315216068687\\
0.7	0.2\\
};
\addlegendentry{obstacle};

\addplot [color=blue,solid,line width=1.0pt]
  table[row sep=crcr]{0	0\\
0.01	1.51841e-006\\
0.02	0.00283129\\
0.03	0.0128297\\
0.04	0.0228296\\
0.05	0.0328295\\
0.06	0.0428295\\
0.07	0.0528295\\
0.08	0.0628295\\
0.09	0.0728295\\
0.1	0.0828294\\
0.11	0.0928294\\
0.12	0.102829\\
0.13	0.112829\\
0.14	0.122829\\
0.15	0.132829\\
0.16	0.142829\\
0.17	0.152829\\
0.18	0.162829\\
0.19	0.172829\\
0.2	0.182829\\
0.21	0.192829\\
0.22	0.202829\\
0.23	0.212829\\
0.24	0.222829\\
0.25	0.232829\\
0.26	0.242829\\
0.27	0.252829\\
0.28	0.262829\\
0.29	0.272829\\
0.3	0.282829\\
0.31	0.292829\\
0.32	0.302829\\
0.33	0.312829\\
0.34	0.322829\\
0.35	0.332829\\
0.36	0.342829\\
0.37	0.351987\\
0.38	0.36\\
0.39	0.367033\\
0.4	0.373205\\
0.41	0.378606\\
0.42	0.383303\\
0.43	0.38735\\
0.44	0.390788\\
0.45	0.393649\\
0.46	0.395959\\
0.47	0.397737\\
0.48	0.398997\\
0.49	0.39975\\
0.5	0.4\\
0.51	0.39975\\
0.52	0.398997\\
0.53	0.397737\\
0.54	0.395959\\
0.55	0.393649\\
0.56	0.390788\\
0.57	0.38735\\
0.58	0.383303\\
0.59	0.378606\\
0.6	0.373205\\
0.61	0.367033\\
0.62	0.37\\
0.63	0.38\\
0.64	0.39\\
0.65	0.4\\
0.66	0.41\\
0.67	0.42\\
0.68	0.43\\
0.69	0.44\\
0.7	0.45\\
0.71	0.46\\
0.72	0.47\\
0.73	0.48\\
0.74	0.49\\
0.75	0.5\\
0.76	0.51\\
0.77	0.52\\
0.78	0.53\\
0.79	0.54\\
0.8	0.55\\
0.81	0.56\\
0.82	0.57\\
0.83	0.58\\
0.84	0.59\\
0.85	0.6\\
0.86	0.61\\
0.87	0.62\\
0.88	0.63\\
0.89	0.64\\
0.9	0.65\\
0.91	0.66\\
0.92	0.67\\
0.93	0.68\\
0.94	0.69\\
0.95	0.7\\
0.96	0.71\\
0.97	0.72\\
0.98	0.73\\
0.99	0.74\\
1	0.75\\
};
\addlegendentry{optimal};

\addplot [color=black,dotted,line width=1.0pt]
  table[row sep=crcr]{0	-0.1\\
0	-0.0499999999999999\\
0.05	-0.0499999999999999\\
0.25	0.0499999999999999\\
0.3	0.1\\
0.5	0.25\\
0.55	0.3\\
0.8	0.5\\
0.85	0.55\\
0.9	0.6\\
0.9	0.65\\
0.95	0.7\\
1	0.75\\
};
\addlegendentry{order 2};

\addplot [color=mycolor1,dash pattern=on 1pt off 3pt on 3pt off 3pt,line width=1.0pt]
  table[row sep=crcr]{0.02	0\\
0.04	0.02\\
0.05	0.02\\
0.06	0.04\\
0.07	0.04\\
0.19	0.14\\
0.21	0.16\\
0.26	0.2\\
0.28	0.22\\
0.4	0.34\\
0.42	0.36\\
0.43	0.36\\
0.44	0.38\\
0.45	0.38\\
0.46	0.4\\
0.47	0.4\\
0.48	0.42\\
0.5	0.26\\
0.55	0.32\\
0.56	0.32\\
0.57	0.34\\
0.62	0.38\\
0.64	0.4\\
0.66	0.42\\
0.68	0.44\\
0.73	0.48\\
0.74	0.5\\
0.79	0.54\\
0.81	0.56\\
0.83	0.58\\
0.95	0.7\\
0.97	0.72\\
};
\addlegendentry{order 6};

\end{axis}
\end{tikzpicture}%
\caption{Comparison between reconstructed trajectories from relaxations of order $2$ and $6$, against the optimal solution for the example of \S\ref{sec:nonconvex}. The latter can be found by using any of the former as a starting point of a local method. }
\label{fig:nonconvex}
\end{figure}
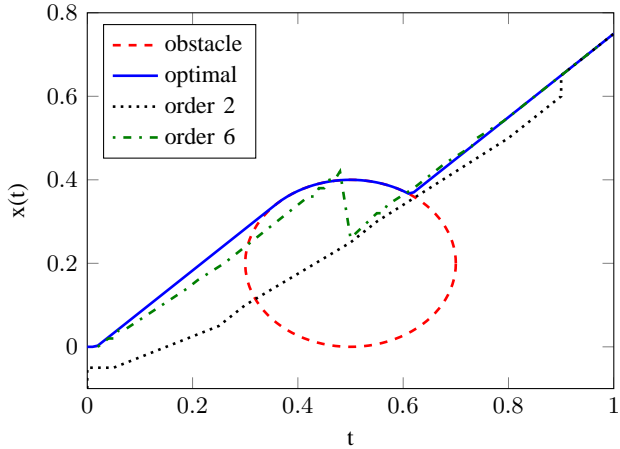

\subsection{Van der Pol oscillator} \label{sec:VDP}
In this last example, we show some alternative use of the proposed approach beyond optimal control. Consider the (uncontrolled) dynamical system $\dot{x} = f(x)$ given by the Van der Pol oscillator:
\begin{equation}
\begin{aligned}
\dot{x}_1 & = x_2 \\
\dot{x}_2 & = - x_1 + (1- x_1^2) x_2.
\end{aligned}
\end{equation}
Suppose one would like to compute its limit set, which is known here to be the union of a stable limit cycle and an unstable fixed point. Following e.g. \cite{gaitsgory}, one can relax this as the measure feasibility test
\begin{equation}
\label{eq:VDP}
\begin{aligned}
\exists \mu ? \; &  \\
\text{s.t.} \; & \forall v \in \R[\xbar]: \;  \langle  \frac{\partial v}{\partial \xbar}  f , \mu \rangle = 0, \\
& \langle 1, \mu \rangle = 1, \\
& \mu \in \PositiveMeasures(\mathbf{X}),
\end{aligned}
\end{equation}
with $\mathbf{X}$ a given compact, basic, semi-algebraic set. This problem can be solved via moment relaxations in a similar procedure as the one presented in \S\ref{sec:OCP}. Applying the atomic extraction procedure for the support of the resulting measure on $\mathbf{X} = [-3,3]^2$, and using moments of the $8$th relaxation, one obtains the results of Fig.~\ref{fig:VDP}. As expected, the extraction procedure locates (within precision of the supplied grid) the stable limit cycle, as well as the unstable equilibrium.

\begin{figure}
\centering
%
%
\begin{tikzpicture}

\begin{axis}[%
width=\figurewidth,
height=\figureheight,
scale only axis,
xmin=-2.5,
xmax=2.5,
xlabel={$x_1$},
ymin=-3,
ymax=3,
ylabel={$x_2$},
legend style={at={(0.97,0.03)},anchor=south east,draw=black,fill=white,legend cell align=left}
]
\addplot [color=blue,solid,line width=1.0pt]
  table[row sep=crcr]{2	0\\
1.9999999899057	-0.000200920631273141\\
1.99999963678758	-0.00120461577712104\\
1.99999032855506	-0.00620045415324861\\
1.99975813090187	-0.0306215837226132\\
1.99430017772298	-0.139756135824967\\
1.97316549077867	-0.276207924850522\\
1.94018778693124	-0.378677771844134\\
1.89819868755362	-0.457987501077065\\
1.84909382485377	-0.522115418063387\\
1.79408879437382	-0.57679381864626\\
1.73391157292987	-0.626140725054396\\
1.66893672878502	-0.673182249124649\\
1.59927393188012	-0.720246017906712\\
1.52482270339431	-0.769247877917971\\
1.44530258100503	-0.821901659317662\\
1.36026506088775	-0.879877565254455\\
1.26909147447489	-0.944927494764828\\
1.17097951951308	-1.01898765200524\\
1.06492057761646	-1.10425943488959\\
0.949670474019255	-1.20325612226013\\
0.823718605471346	-1.31878038030007\\
0.685265644146051	-1.45375832217905\\
0.532230498651519	-1.61078956649084\\
0.362325997916517	-1.79117278577123\\
0.173272611804375	-1.99305312809854\\
-0.0367426734663199	-2.2083239487402\\
-0.268233244813862	-2.41832466306293\\
-0.519140740850931	-2.58980269368264\\
-0.783335056286548	-2.67536103775528\\
-1.04964997443881	-2.62471792452937\\
-1.30270813237666	-2.40854773895574\\
-1.52629502526207	-2.04209010116726\\
-1.70811163482096	-1.58578956765231\\
-1.84306262069996	-1.11712403940138\\
-1.93318413869992	-0.696737548424418\\
-1.98501033700136	-0.353444020451387\\
-2.0065312801296	-0.0894467265939232\\
-2.00515080139525	0.107127621235671\\
-1.98682208897722	0.252154484753841\\
-1.95594688418396	0.360263694248676\\
-1.9156087558697	0.443086625783403\\
-1.8678826142508	0.509235720396361\\
-1.81410946264729	0.56489031276992\\
-1.75510595468458	0.614462917036565\\
-1.69131181631655	0.661161761129532\\
-1.62288788487803	0.707416759702997\\
-1.54977762181277	0.755189080725682\\
-1.47174217910753	0.806195654810842\\
-1.38837607803779	0.862076492124734\\
-1.29910812719586	0.924525352376804\\
-1.20319054533544	0.995396407196151\\
-1.09967841266892	1.07679083075617\\
-0.987401727265906	1.17111553940977\\
-0.864933983414029	1.28108737664421\\
-0.730565353855566	1.40962307325449\\
-0.582297117236707	1.55949860748579\\
-0.417889798676502	1.7325719883109\\
-0.235023771076994	1.92824848435132\\
-0.0316675770946823	2.14079989414637\\
0.193218952514642	2.35537294768973\\
0.43854548017627	2.54355918679552\\
0.699609664757352	2.66178678141599\\
0.966842100243951	2.65857435514674\\
1.22593385814966	2.49499656554952\\
1.46042760047854	2.17076821784409\\
1.65630681969951	1.73412394948468\\
1.80605640283034	1.2612819417728\\
1.90968415183101	0.820972540165882\\
1.97266916286046	0.452173562047778\\
2.00283100054689	0.164143896975973\\
2.00790403027736	-0.0519012583303328\\
2.00636706579025	-0.0904588617115785\\
};
\addlegendentry{cycle};

\addplot [color=red,only marks,mark=o,mark options={solid}]
  table[row sep=crcr]{-2.01	-0.21\\
-2.01	-0.18\\
-2.01	0.12\\
-2.01	0.15\\
-1.98	0.21\\
-1.98	0.24\\
-1.95	-0.6\\
-1.92	-0.75\\
-1.92	-0.72\\
-1.92	0.39\\
-1.92	0.42\\
-1.89	-0.87\\
-1.89	0.45\\
-1.86	0.48\\
-1.83	0.54\\
-1.8	-1.26\\
-1.71	-1.59\\
-1.71	0.66\\
-1.65	-1.74\\
-1.62	0.72\\
-1.59	0.75\\
-1.53	0.78\\
-1.5	-2.1\\
-1.5	0.81\\
-1.44	0.84\\
-1.38	0.87\\
-1.35	-2.34\\
-1.29	-2.43\\
-1.26	-2.46\\
-1.14	-2.58\\
-1.08	-2.61\\
-1.05	1.11\\
-1.02	1.14\\
-0.99	1.17\\
-0.93	-2.67\\
-0.9	-2.67\\
-0.81	1.32\\
-0.78	1.35\\
-0.75	-2.67\\
-0.75	1.38\\
-0.72	-2.67\\
-0.72	1.41\\
-0.51	-2.58\\
-0.45	-2.55\\
-0.39	1.77\\
-0.36	-2.49\\
-0.36	1.8\\
-0.33	-2.46\\
-0.33	1.83\\
-0.3	1.86\\
-0.27	1.89\\
-0.24	-2.4\\
-0.24	1.92\\
-0.21	-2.37\\
-0.18	-2.34\\
-0.15	-2.31\\
0	0\\
0.15	2.31\\
0.18	2.34\\
0.21	2.37\\
0.24	-1.92\\
0.24	2.4\\
0.27	-1.89\\
0.3	-1.86\\
0.33	-1.83\\
0.33	2.46\\
0.36	-1.8\\
0.36	2.49\\
0.39	-1.77\\
0.42	-1.74\\
0.51	2.58\\
0.63	2.64\\
0.72	-1.41\\
0.75	-1.38\\
0.78	-1.35\\
0.81	-1.32\\
0.93	-1.2\\
0.93	2.67\\
0.96	-1.17\\
0.99	-1.17\\
1.02	-1.14\\
1.05	-1.11\\
1.08	2.61\\
1.26	2.46\\
1.29	2.43\\
1.32	2.4\\
1.35	-0.9\\
1.35	2.34\\
1.44	-0.84\\
1.5	-0.81\\
1.53	-0.78\\
1.53	2.04\\
1.59	-0.75\\
1.62	-0.72\\
1.62	1.83\\
1.65	1.74\\
1.68	1.68\\
1.71	1.59\\
1.74	-0.63\\
1.77	-0.6\\
1.77	1.38\\
1.83	-0.54\\
1.89	-0.45\\
1.89	0.9\\
1.92	-0.42\\
1.92	0.72\\
1.92	0.75\\
1.95	-0.33\\
1.95	0.6\\
1.98	-0.24\\
1.98	-0.21\\
2.01	-0.15\\
2.01	-0.12\\
2.01	0.18\\
};
\addlegendentry{atomic};

\end{axis}

\end{tikzpicture}%
\caption{Reconstruction for the example of \S\ref{sec:VDP}. Simulated limit cycle versus atomic approximation.}
\label{fig:VDP}
\end{figure}
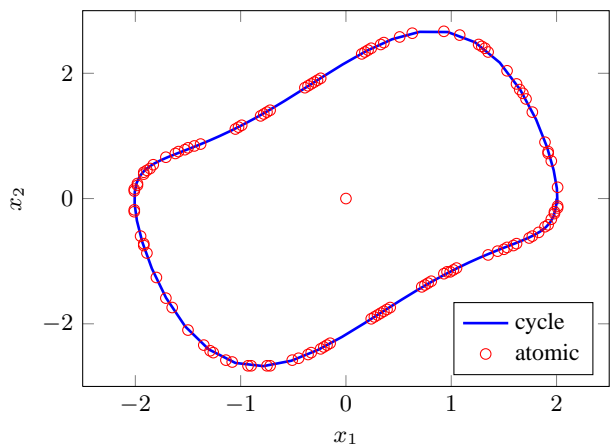

\section{Perspectives}

In this paper, we present a method for recovering trajectories from the moments of occupation measures. This allows for a numerical method solving optimal control problem globally, in a ``black box'' fashion. First, moment relaxations as presented in \cite{Lasserre2008NLOCP} are computed to obtain approximate truncated moment sequences. Then, the method outlined here solves the inverse problem to obtain approximate global solutions. Those solutions are then fed to a local method to guarantee admissibility of the process.

Numerically, the inverse problem \eqref{eq:approxLPexplicit} possesses a simple structure, as one tries to minimize the distance in a finite-dimensional space, as measured by the supremum norm, from a point to a linear subspace.  As such, it is expected that one could use dedicated solvers, using large scale, problem-specific routines to outperform the simple linear programming method considered here. See e.g. \cite{Nesterov} for an introduction on the subject. This opens the possibility of identifying approximate trajectories and controls in one pass, using all information available from moments.

Finally, we also presented the approach for bounded control for ease of exposition, but the method holds as well for modal occupation measures of switched systems \cite{Henrion2013Switched} and impulsive occupation measures \cite{impulseTAC}. This demonstrates the high flexibility of the approach.

\section*{Acknowledgments}

This work was supported by the Engineering and Physical Sciences Research Council under Grant EP/G066477/1, and benefited from discussions with Jean-Bernard Lasserre and Pierre Martinon.

\end{document}